\documentclass[a4paper,12pt]{article}

\usepackage{mathtools}

\usepackage{amstext,amssymb,amsmath,amsthm}
\usepackage{relsize,xspace,ifthen}
\usepackage[T1]{fontenc}
\usepackage{microtype}
\usepackage{enumerate}
\usepackage[utf8]{inputenc}
\usepackage{tikz}
\usetikzlibrary{arrows}
\usetikzlibrary{shapes.geometric}
\usetikzlibrary{calc, decorations.markings}
\usepackage{hyperref}
\usepackage{comment}

\usepackage{mathrsfs}

\usepackage[font=small]{caption}

\usepackage[T1]{fontenc}
    \usepackage{ae}
    \usepackage{aecompl}

\newenvironment{cenv}{\begin{list}{}{%
      \setlength{\labelwidth}{1.5em}%
      \setlength{\leftmargin}{\labelwidth}%
      \addtolength{\leftmargin}{\labelsep}%
      \setlength{\listparindent}{0em}%
      \setlength{\topsep}{10pt}%
      \setlength{\itemsep}{5pt}%
      \setlength{\parsep}{0pt}%
    }
  }{
  \end{list}
}

\newcounter{claimcounter}

\usepackage{latexsym}
\usepackage{nicefrac}

\usepackage{algorithm}
\usepackage{algpseudocode}
\usepackage{eqparbox}

\newtheorem{theorem}{Theorem}
\newtheorem{lemma}{Lemma}

\newtheorem{rem}{Remark}
\let\oldrem\rem
\renewcommand{\rem}{\oldrem\normalfont}

\numberwithin{equation}{section}
\numberwithin{theorem}{section}
\numberwithin{lemma}{section}

 \newcommand{\DDD}{\mathcal{D}}
 \newcommand{\FFF}{\mathcal{F}}

 \newcommand{\LLL}{\mathcal{L}}
\newcommand{\MMM}{\ensuremath{\mathcal{M}}} 
 \newcommand{\PPP}{\mathcal{P}}
 
 \newcommand{\TTT}{\mathcal{T}}
 
 \newcommand{\XXX}{\mathcal{X}}
 \newcommand{\ZZZ}{\mathcal{Z}}

\renewcommand{\phi}{\varphi}
\renewcommand{\epsilon}{\varepsilon}

\newcommand{\ie}{i.e.\@\xspace}

\DeclareMathOperator{\tw}{tw}
\DeclareMathOperator{\mw}{mw}
\DeclareMathOperator{\lw}{lw}

\DeclareMathOperator{\pw}{pw}

\title{Cops, Robber and Medianwidth Parameters}
\author{Konstantinos Stavropoulos \\ \small{RWTH Aachen University} \\ \small{stavropoulos@informatik.rwth-aachen.de}}
\date{}

\begin{document}

\maketitle

\begin{abstract}
In previous work, we introduced \emph{median decompositions} of graphs, a generalisation of tree decompositions where a graph can be modelled after any median graph,
along with a hierarchy of $i$-medianwidth parameters $(\mw_i)_{i\geq 1}$ starting from
treewidth and converging to the clique number. 

We introduce another graph parameter based on the concept of median decompositions, to be called \emph{$i$-latticewidth} and denoted by $\lw_i$, 
for which we restrict the modelling median graph of a decomposition to be isometrically embeddable into the Cartesian product of $i$ paths. 
The sequence $(\lw_i)_{i\geq 1}$ gives rise to a hierarchy of parameters starting from pathwidth and converging to the clique number.
We characterise the $i$-latticewidth of a graph in terms of maximal intersections of bags of $i$ path decompositions of the graph. 

We study a generalisation of the classical \emph{Cops and Robber game}, where the robber plays against not just one, but $i$ cop players. 
Depending on whether the robber is visible or not, we show a direct connection to $i$-medianwidth or $i$-latticewidth, respectively.
\end{abstract}

\section{Introduction}\label{sec:intro}

The concept of modelling a graph after simpler ones has proven to be a powerful and fruitful 
approach in graph theory. For example, tree and path decompositions, where one models a graph after trees or paths, 
have been fundamental in the study of graph classes excluding a fixed minor~\cite{robertson2003graph}.
Their respective width parameters, \emph{treewidth} and \emph{pathwidth}, have been extensively studied in various contexts throughout the literature~\cite{reed1997tree,Bodlaender93,Bodlaender05}.

A median graph is a connected graph, 
such that for any three vertices $u,v,w$ there is exactly one vertex $x$ that lies simultaneously
on a shortest $(u,v)$-path, a shortest $(v,w)$-path and a shortest $(w,u)$-path. Examples of median graphs are trees, grids and the $i$-dimensional hypercube $Q_i$, for every $i\geq 1$. 

A subset $S$ of vertices of a graph is \emph{(geodesically) convex} if for every pair of vertices in $S$, all shortest paths between them only contain vertices in $S$.
Making use of the observation that a subset of the vertices of a tree is convex if and only if it induces a connected subgraph, 
the notion of median decompositions was introduced in~\cite{stavropoulos2015medianwidth}, 
which models a graph not only after trees or paths, but after any median graph. 
It was proven there that the corresponding width parameter \emph{medianwidth} $\mw(G)$ coincides with the clique number
$\omega(G)$ of a graph $G$.

Median decompositions broaden substantially the perspective with which we can view graphs:
they provide a means to see every graph as a multidimensional object.
Every median graph can be isometrically embedded into the Cartesian product of a finite number of trees.  
When we restrict to median decompositions whose underlying median graph must be isometrically embeddable into the Cartesian product of $i$ trees, 
the respective medianwidth parameter of $G$ is called $i$-medianwidth  and denoted by $\mw_i(G)$. 
Then $\mw_i(G)$ can be seen as the largest ``intersection'' of the best choice of $i$ tree decompositions of the graph $G$~\cite{stavropoulos2015medianwidth}. 
By their definition, the invariants $\mw_i$ form a non-increasing sequence starting from treewdith and converging to the clique number:
$$\tw(G)+1=\mw_1(G)\geq \mw_2(G)\geq \dots \geq \mw_{\infty}(G)=\omega(G).$$

Instead of considering median decompositions whose underlying median graph can be isometrically embedded into the Cartesian product of $i$ trees, in Section~\ref{sec:lat} we study
medianwidth parameters for which we consider median decompositions whose underlying median graph must be isometrically embeddable into the Cartesian product of $i$ \emph{paths}.
For $i\geq1$, the corresponding width parameters, to be called $i$-latticewidth $\lw_i(G)$, will give rise to a sequence converging to the clique number and starting from pathwidth this time around:
$$\pw(G)+1=\lw_1(G)\geq \lw_2(G)\geq \dots \geq \lw_{\infty}(G)=\omega(G).$$
As in the case of $\mw_i$, by considering complete multipartite graphs, we see that this hierarchy of the $\lw_i$ parameters has stronlgy distinguished levels:
for $i<i'$, graph classes of bounded $i'$-latticewidth can have unbounded $i$-latticewidth. 
Lastly, we provide a characterisation of $i$-latticewidth in terms of path decompositions: 
we prove that it corresponds to the largest ``intersection'' of the best choice of $i$ path decompositions of the graph.


A large variety of width parameters for graphs are characterised through so-called \emph{search games}, introduced by Parsons and Petrov in~\cite{parsons1978pursuit,parsons1978search,petrov1982problem}. 
A set of searchers and a fugitive move on a graph according to some rules specified by the game. The goal of the searchers is to capture the fugitive, whose goal is to avoid capture. 
Different variants of the rules according to which the searchers and the fugitive move, give rise to games that characterise related width parameters, 
often otherwise introduced and appearing in different contexts. These game characterisations provide a better understanding of the parameters. For a survey on search games, see~\cite{fomin2008annotated}.

Treewidth and pathwidth are known to be characterised by appropriate variations of the \emph{Cops and Robber game} --- sometimes seen as \emph{helicopter Cops and Robber game} in the literature. The game is played on a finite, undirected graph $G$ by the cop player, 
who controls $k$ cops, and the robber player. The robber stands on a vertex of $G$ and can run arbitrarily fast through a path of $G$ to any other vertex, 
as long as there are no cops standing on the vertices of the path. Each of the $k$ cops either stands on a vertex of $G$ or is in a helicopter in the air. 
The cop player tries to capture the robber by landing a cop with a helicopter on the vertex where the robber stands and the robber
tries never to be captured. The robber sees where each of the $k$ cops stands or if they are going to land on a vertex of $G$ and can move arbitrarily fast to another vertex to evade capture while
some of the cops are still in the air.

While the robber can always see the cops at any point of the game, there are two forms of the game with respect to the information available to the cop player. 
In the first variation, the cop player can see the robber at all times and tries to surround her in some corner of the graph. This version of the game characterises the treewidth of $G$
in the sense that the cop player has a winning strategy with at most $k$ cops if and only if $\tw(G) \leq k-1$~\cite{seymour1993graph}. 
In the second variation of the game, the robber is invisible to the cop player so he has to search the graph in a more methodical way. 
In this version of the game, the cop player can always win with at most $k$ cops if and only if $\pw(G)\leq k-1$~\cite{bienstock1991quickly,kirousis1985interval,lapaugh1993recontamination}.

In light of the interplay of $i$-medianwith with intersections of bags of tree decompositions and that of $i$-latticewidth with intersections of bags of path decompositions of the graph, 
we introduce the \emph{$i$-Cops and Robber game} which turns out to be closely connected to the respective medianwidth and latticewidth parameters. The robber player now plays against $i$ cop players
which need to cooperate in order to capture the robber with the least ``cooperation'' possible (to be explained later).
Every cop player has at his disposal a team of $|V(G)|$ cops, each of which can stand on a vertex or move with a helicopter in the air. 
Cop teams are ``undercover'' though, meaning that they are invisible to the other cop teams.

But the robber is very powerful: she can see all $i$ cop teams and how they move at all times and additionally, 
the only way that the robber can be fully caught is by having a cop of every team on the vertex currently 
occupied by the robber. 
Moreover, she has a way to restrict their movement by selecting each time a cop team that she allows to move, while
forcing all the other cop teams to remain still during said move. In other words, the cop teams move one at a time, with any order the robber prefers.

On the other hand, each cop team can also restrict the robber. If one cop player manages to somehow catch the robber (notice that the robber would still not be completely captured), 
then the cops of that team lock down on the vertices they currently occupy and they trap the robber in the following sense: from then on, she is allowed to only move to vertices occupied by said cop team and disabled from choosing that particular cop team to move again for the rest of the game. Moreover, when the robber allows a cop team to move, 
during the time the respective cop player moves some of his cops with helicopters to some other vertices, 
the robber can then move through a path of $G$ to any other vertex, as long as there are no cops of a team that has not trapped her yet standing on the vertices of the path.

The \emph{cooperation} of the cop players is the maximum number of vertices simultaneously occupied by a cop of every team at any point of the game. In case the cop players have
a winning strategy to always catch the robber with cooperation at most $k$, we say that \emph{$i$ cop players can search the graph with cooperation at most $k$}. Moreover, we say that the cop players
can capture the robber \emph{monotonely} if with each of their moves the robber space of available escape options always shrinks. 

We study two variations of this game as well: one where the robber is \emph{visible} and one where the robber is \emph{invisible} to each cop player. 
When the robber is visible (respectively \emph{invisible}),
we say that the cop players search the graph \emph{with vision} (respectively \emph{without vision}).

Note that for $i=1$ the game described above becomes the classical Cops and Robber game where only one cop player searches the graph, because in that case the cooperation degenerates to just being 
the size of the cop team the cop player can use. We already mentioned that $1$-medianwidth corresponds to treewidth and $1$-latticewidth corresponds to pathwidth, 
which are both characterised by the classical Cops and Robber game 
depending on the visibility of the robber. In Section~\ref{sec:game} we extend this connection between Cops and Robber games and width parameters, and 
we show that $i$ cop players \emph{with vision} can \emph{monotonely} search a graph $G$ with cooperation at most $k$ if and only if $\mw_i(G)\leq k$. Similarly,
we show that $i$ cop players \emph{without vision} can \emph{monotonely} search a graph $G$ with cooperation at most $k$ if and only if $\lw_i(G)\leq k$. 
To our knowledge, this is also the first instance of a search game played between a single fugitive player against a team of many search players connected to a width parameter of graphs.


\section{Median Graphs and Median Decompositions}\label{sec:prelim}

Our notation from graph theory is standard, we defer the reader to~\cite{Diestel05} for the background. For a detailed view on median graphs, the reader can refer to books
\cite{feder1995stable,imrich2000product,van1993theory} and papers~\cite{bandelt1983median,klavzar1999median}, or a general survey on metric graph theory and geometry~\cite{bandelt2008metric}.
In this paper, every graph we consider is finite, undirected and simple.

For $u,v \in V(G)$, a \emph{$(u,v)$-geodesic} is a shortest $(u,v)$-path. A path $P$ in $G$ is a \emph{geodesic} if there are vertices $u,v$ such that $P$ is a $(u,v)$-geodesic. 

The \emph{interval} $I(u,v)$ consists of all vertices lying on a $(u,v)$-geodesic, namely
$$I(u,v)=\{x\in V(G) \mid d(u,v)=d(u,x)+d(x,v)\}.$$ 

A graph $G$ is called \emph{median} if it is connected and 
for any three vertices $u,v,w \in V(G)$ there is a unique vertex $x$, called the \emph{median} of $u,v,w$, 
that lies simultaneously on a $(u,v)$-geodesic, $(v,w)$-geodesic and a $(w,u)$-geodesic. In other words, $G$ is median if $|I(u,v) \cap I(v,w) \cap I(w,u)|=1$, for every three vertices $u,v,w$.

A set $S\subseteq V(G)$ is called \emph{geodesically convex} or just \emph{convex} if for every $u,v\in S$, $I(u,v)\subseteq S$.
By definition, convex sets are connected. It is easy to see that the intersection of convex sets is again convex. 
Note that the induced subgraphs corresponding to convex sets of median graphs are also median graphs.

The \emph{$i$-dimensional hypercube} or \emph{$i$-cube} $Q_i$, $i\geq 1$, is the graph with vertex set $\{0, 1\}^i$,
two vertices being adjacent if the corresponding tuples differ in precisely one position.
The hypercubes are also the only regular median graphs~\cite{mulder1980n}.

The Cartesian product $G\Box H$ of graphs $G$ and $H$ is the graph with vertex set $V(G)\times V(H)$,
in which vertices $(a,x)$ and $(b,y)$ are adjacent whenever $ab\in E(G)$ and $x=y$, or $a=b$
and $xy\in E(H)$. The Cartesian product is associative and commutative with $K_1$ as its unit.
Note that the Cartesian product of $i$ copies of $K_2=Q_1$ is an equivalent definition of the $i$-cube $Q_i$.

In Cartesian products of median graphs, medians of vertices can be seen to correspond to the tuple of the medians in every factor of the product. 
The following lemma~is folklore.

\begin{lemma}\label{cartmed}
 Let $G=\Box_{i=1}^kG_i$, where $G_i$ is median for every $i=1,\ldots,k$. Then $G$ is also median, whose convex sets are precicely the sets $C=\Box_{i=1}^kC_i$,
 where $C_i$ is a convex subset of $G_i$.
\end{lemma}

A graph $G$ is a \emph{convex amalgam} of
two graphs $G_1$ and $G_2$ (along $G_1\cap G_2$) if $G_1$ and $G_2$ constitute two intersecting induced convex subgraphs of
$G$ whose union is all of $G$.

A (necessarily induced) subgraph $H$ of a graph $G$ is a \emph{retract} of $G$, if there is a map $r:V(G)\rightarrow V(H)$ that maps each edge of $G$ to an edge of $H$,
and fixes $H$, \ie, $r(v) =v$ for every $v\in V(H)$. Median graphs are easily seen to be closed under retraction, 
and since they include the $i$-cubes, every retract of a hypercube is a median graph. The inverse is also true: 

\begin{theorem}\cite{bandelt1984retracts,isbell1980median,van1983matching}\label{retract}
 A graph $G$ is median if and only if it is the retract of a hypercube. Every median graph with more than two vertices is either a Cartesian
product or a convex amalgam of proper median subgraphs.
\end{theorem}

A corollary of Theorem~\ref{retract} is that median graphs are bipartite graphs. 
On the other hand, not all bipartite graphs are median: the cycle graph on $2k$ vertices $C_{2k}$ is not median for $k\geq 3$,
because there always three vertices without a median; $K_{2,3}$ is also not median, since the three vertices of the one part have two medians, exactly the vertices of the other part.   

A graph $H$ is \emph{isometrically embeddable} into a graph $G$ if there is a mapping $\phi:V(H)\rightarrow V(G)$ such that $d_G(\phi(u), \phi(v))=d_H(u, v)$ for any vertices $u,v \in H$.
Isometric subgraphs of hypercubes are called \emph{partial cubes}. Retracts of graphs are isometric subgraphs, hence median graphs are partial cubes,
but not every partial cube is a median graph: $C_6$ is an isometric subgraph of $Q_3$, but not a median graph.

For a connected graph and an edge $ab$ of $G$ we denote
\begin{itemize} 
 \item $W_{ab}=\{v\in V(G) \mid d(v,a)<d(v,b)\},$
 \item $U_{ab}=W_{ab}\cap N_G(W_{ba}).$
\end{itemize}

\noindent
Sets of the graph that are $W_{ab}$ for some edge $ab$ will be called \emph{$W$-sets} and similarly we define \emph{$U$-sets}. If $U_{ab}=W_{ab}$ for some edge $ab$, 
we call the set $U_{ab}$ a peripheral set of the graph.
Note that if $G$ is a bipartite graph, then $V(G)=W_{ab}\cup W_{ba}$ and $W_{ab}\cap W_{ba}=\emptyset$ is true for any edge $ab$.  
If $G$ is a median graph, it is easy to see that $W$-sets and $U$-sets are convex sets of $G$. 

Edges $e=xy$ and $f=uv$ of a graph $G$ are in the \emph{Djokovic-Winkler} relation $\Theta$~\cite{djokovic1973distance,winkler1984isometric} if $d_G(x,u)+d_G(y,v)\neq d_G(x,v)+d_G(y,u)$. 
Relation $\Theta$ is reflexive and symmetric. If $G$ is bipartite, then $\Theta$ can be defined as follows: $e=xy$ and $f=uv$ are in relation $\Theta$ if $d(x,u)=d(y,v)$ and $d(x,v)=d(y,u)$. 
Winkler~\cite{winkler1984isometric} proved that on bipartite graphs relation $\Theta$ is transitive if and only if it is a partial cube and so, 
by Theorem \ref{retract} it is an equivalence relation on the edge set of every median graph, whose classes we call $\Theta$-classes. 

The following lemma summarises some properties of the $\Theta$-classes of a median graph:

\begin{lemma}\cite{imrich2000product}\label{theta}
 Let $G$ be a median graph and for an edge $ab$, let $F_{ab}=F_{ba}$ denote the set of edges between $W_{ab}$ and $W_{ba}$. Then the following are true:
 \begin{enumerate}
  \item $F_{ab}$ is a matching of $G$.
  \item $F_{ab}$ is a minimal cut of $G$.
  \item A set $F\subseteq E(G)$ is a $\Theta$-class of $G$ if and only if $F=F_{ab}$ for some edge $ab\in E(G)$. 
 \end{enumerate}

\end{lemma}
 
Characterisations of median graphs in terms of \emph{expansions} can be seen in~\cite{mulder1978structure,mulder1980interval,mulder1990expansion}.

Concluding with the general properties of median graphs, a family of sets $\FFF$ on a universe $U$ has the \emph{Helly property}, 
if every finite subfamily of $\FFF$ with pairwise-intersecting sets, has a non-empty total intersection.

\begin{lemma}\cite{imrich2000product}\label{helly}
The convex sets of a median graph $G$ have the Helly property. 
\end{lemma}

A \emph{tree decomposition} $\DDD$ of a graph $G$ is a pair $(T,\ZZZ)$, where
$T$ is a tree and $\ZZZ=(Z_t)_{t\in V(T)}$ is a family of subsets of $V(G)$
(called bags) such that 
\begin{enumerate}[(i)]
 \item[(T1)] for every edge $uv\in E(G)$ there exists $t \in V(T)$ with $u,v \in Z_t$, 
 \item[(T2)] for every $v\in V(G)$, the set $Z^{-1}(v):=\{t\in V(T) \mid v\in Z_t\}$ is a non-empty connected subgraph (a subtree) of $T$.
\end{enumerate} 
The \emph{width} of a tree decomposition $\DDD=(T,\ZZZ)$ is the number $$\max\{|Z_t|-1 \mid t\in V(T)\}.$$
Let $\TTT^G$ be the set of all tree decompositions of $G$. The \emph{treewidth} $\tw(G)$ of $G$ is the least width of any tree decomposition of $G$, namely
$$\tw(G):=\min\limits_{\DDD\in \TTT^G}\max\{|Z_t|-1 \mid t\in V(T)\}.$$ 
When $T$ is a path $P$, we call $(P,\ZZZ)$ a path decomposition. Its width is defined as in tree decompositions. Let $\PPP^G$ be the set of all path decompositions of $G$. 
The \emph{pathwidth} $\pw(G)$ of $G$ is the least width of any path decomposition in $\PPP^G$.

A pair $(A,B)$ is a \emph{separation} of $G$ if $A\cup B=V(G)$ and $G$ has no edge between $A\setminus B$ and $B\setminus A$. For a separation $(A,B)$, we say that $A\cap B$ separates $A$ from $B$. Let the clique number $\omega(G)$ be the size of the 
largest complete subgraph of $G$. 

We assume familiarity with the basic theory of tree decompositions as in~\cite{Diestel05} or~\cite{reed1997tree}. In the next lemma, we briefly state some of the most important well-known properties of tree decompositions.

\begin{lemma}\label{twprop}
 Let $\DDD=(T,\ZZZ) \in \TTT^G$.
 \begin{enumerate}[(i)]
  \item For every $H\subseteq G$, the pair $(T, (Z_t\cap V(H))_{t\in T})$ is a tree decomposition of $H$, so that $\tw(H)\leq \tw(G)$.
  \item Any complete subgraph of $G$ is contained in some bag of $\DDD$, hence $\omega(G)\leq \tw(G)+1$.
  \item For every edge $t_1t_2$ of $T$, $Z_{t_1}\cap Z_{t_2}$ separates $W_1:=\bigcup_{t\in T_1} Z_t$ from $W_2:=\bigcup_{t\in T_2} Z_t$,
  where $T_1,T_2$ are the components of $T-t_1t_2$, with $t_1 \in T_1$ and $t_2 \in T_2$.
 \end{enumerate}

\end{lemma}

\noindent
Inspired by the nature of (T2) and the fact that a subgraph of a tree is convex if and only if it is connected, 
general median decompositions were introduced in~\cite{stavropoulos2015medianwidth}.

A \emph{median decomposition} $\DDD$ of a graph $G$ is a pair $(M,\XXX)$, where
$M$ is a median graph and $\XXX=(X_a)_{a\in V(M)}$ is a family of subsets of $V(G)$
(called bags) such that 
\begin{itemize}
 \item[(M1)] for every edge $uv\in E(G)$ there exists $a \in V(M)$ with $u,v \in X_a$, 
 \item[(M2)] for every $v\in V(G)$, the set $X^{-1}(v):=\{a\in V(M) \mid v\in X_a\}$ is a non-empty convex subgraph of $M$.
\end{itemize} 
The \emph{width} of a median decomposition $\DDD=(T,\XXX)$ is the number $$\max\{|X_a| \mid a\in V(M)\}.$$ 

\noindent
Let $\MMM^G$ be the set of all median decompositions of $G$. The \emph{medianwidth} $\mw(G)$ of $G$ is the least width of any median decomposition of $G$:
$$\mw(G):=\min\limits_{\DDD\in \MMM^G}\max\{|X_a| \mid a\in V(M)\}.$$ 
\noindent
Since $\TTT^G \subseteq \MMM^G$, by definition of $\mw(G)$ we have $\mw(G)\leq tw(G)+1$.

Let us fast go through the main properties of median decompositions.
For the lemmata that follow, $\DDD=(M,\XXX) \in \MMM^G$ is a median decomposition of a graph $G$. 

\begin{lemma}\cite{stavropoulos2015medianwidth}\label{submw}
 For every $H\subseteq G$, $(M, (X_a\cap V(H))_{a\in M})$ is a median decomposition of $H$, hence $\mw(H)\leq \mw(G)$.
\end{lemma}

The Helly property of the convex sets of median graphs was the second reason that indicated that median decompositions seem to be a natural notion. 
It is what allowed the proof of a direct analogue of Lemma~\ref{twprop} (ii).

\begin{lemma}\cite{stavropoulos2015medianwidth}\label{omega<mw}
 Any complete subgraph of $G$ is contained in some bag of $\DDD$. In particular, $\omega(G)\leq \mw(G)$.
\end{lemma}

For a median decomposition $(M,\XXX)$ and a minimal cut $F\subseteq E(M)$ of $M$ that separates $V(M)$ into $W_1$ and $W_2$, let $U_i$ be the vertices of $W_i$ adjacent to edges of $F$,
and let $Y_i:= \bigcup_{x\in W_i}X_x$, $Z_i:= \bigcup_{x\in U_i}X_x$, where $i=1,2$. 
The analogue of Lemma~\ref{twprop}(iii) says that minimal cuts of $M$ correspond to separations of $G$.

\begin{lemma}\cite{stavropoulos2015medianwidth}\label{mincut}
 For every minimal cut $F$ of $M$ and $Y_i, Z_i$, $i=1,2$, defined as above, $Z_1\cap Z_2$ separates $Y_1$ from $Y_2$.  
\end{lemma}

Recall that by Lemma~\ref{theta}, for an edge $ab$ of $M$, the $\Theta$-class $F_{ab}$ is a minimal cut of $M$. Denote $Y_{ab}:= \bigcup_{x\in W_{ab}}X_x$ and $Z_{ab}:= \bigcup_{x\in U_{ab}}X_x$. 
We will refer to them as the \emph{$Y$-sets} and \emph{$Z$-sets} of a median decomposition $\DDD$.
Note that the $Y$-sets and $Z$-sets are subsets of the decomposed graph $G$, while the $W$-sets and $U$-sets are subsets of the median graph $M$ of the decomposition.
A special case of Lemma~\ref{mincut} is the following more specific analogue of Lemma~\ref{twprop}(iii).

\begin{lemma}\cite{stavropoulos2015medianwidth}\label{sep}
 For every edge $ab$ of $M$, $Z_{ab}\cap Z_{ba}$ separates $Y_{ab}$ from $Y_{ba}$.\qed
\end{lemma}


Finally, it turns out that medianwidth is exactly the clique number of a graph.


\begin{figure}[t]
 \centering
 
\begin{tikzpicture}[
  dot/.style={draw,fill=black,circle,inner sep=0.5pt}
  ]

 \foreach \i in {1,...,5} {
    \node[dot,xshift=-100,yshift=-10,label=18+\i*72:\tiny{$\i$}] (\i) at (18+\i*72:1) {};
}

\draw (1) -- (2);
\draw (2) -- (3);
\draw (3) -- (4);
\draw (4) -- (5);
\draw (5) -- (1);
 
\node[dot,fill=white] (12) at (0,0) {12};
\node[draw,circle,inner sep=2pt,fill=white] (11) at (1,0) {1};
\node[dot,fill=white] (15) at (2,0) {15};
\node[dot,fill=white] (45) at (2,-1) {45};
\node[dot,fill=white] (34) at (1,-1) {34};
\node[dot,fill=white] (23) at (0,-1) {23};

\draw (12) -- (11);
\draw (11) -- (15);
\draw (15) -- (45);
\draw (45) -- (34);
\draw (34) -- (23);
\draw (12) -- (23);
\draw (11) -- (34);

\end{tikzpicture}
\caption{A median decomposition of $C_5$ of width 2.}
\label{fig:c5}
\end{figure}
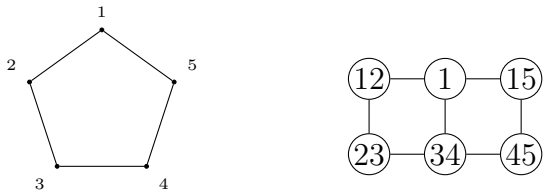


\begin{theorem}\cite{stavropoulos2015medianwidth}\label{omega=mw}
 For any graph $G$, $\mw(G)=\omega(G).$
\end{theorem}

\section{The $i$-Latticewidth of Graphs}\label{sec:lat}

In~\cite{stavropoulos2015medianwidth}, a notion of dimension, called the \emph{tree dimension}, was used to restrict the median graphs $M$ allowed as a model of a median decomposition of a graph. 
The tree dimension of a graph $M$ is the minimum $k$ such that $M$ has an isometric embedding into a Cartesian product of $k$ trees 
and it is well-known that median graphs have finite tree dimension~\cite{ovchinnikov2011graphs}.

Fon an $i\geq 1$, an \emph{$i$-median decomposition} of $G$ is a median decomposition $\DDD=(M,\XXX)$ satisfying (M1),(M2), where $M$ is a median graph of tree dimension at most $i$.
We denote the set of $i$-median decompositions of $G$ as $\MMM^G_i$. The \emph{$i$-medianwidth} $\mw_i(G)$ of $G$ is the least width of any $i$-median decomposition of $G$:
$$\mw_i(G):=\min\limits_{\DDD\in \MMM^G_i}\max\{|X_a| \mid a\in V(M)\}.$$ 
Since trees are exactly the median graphs of tree dimension $1$, the $1$-median decompositions are the tree decompositions of $G$, therefore $\mw_1(G)=\tw(G)+1$. 
This gives rise to a non-increasing sequence of invariants formed by $\mw_i$:
$$\tw(G)+1=\mw_1(G)\geq \mw_2(G)\geq \dots \geq \mw(G)=\omega(G).$$

\noindent
One of the main results of~\cite{stavropoulos2015medianwidth} was that the $i$-medianwidth of a graph can be seen as the ``best intersection of $i$ tree-decompositions of the graph``.
In the following theorem, when we denote tree decompositions with $\DDD^j$, we mean $\DDD^j=(T^j,\ZZZ^j)$. More precisely:

\begin{theorem}\label{imwchar}
 For any graph $G$ and any integer $i\geq 1$, $$\mw_i(G)=\min\limits_{\DDD^1,\ldots,\DDD^i\in \TTT^G}\max\{|\bigcap_{j=1}^iZ^j_{t_j}|\mid t_j\in V(T^j)\}.$$
\end{theorem}

A $k$-lattice graph is a graph obtained by the Cartesian product of $k$ paths. In this section, we turn into another notion of dimension for median graphs, the \emph{lattice dimension}, 
namely the minimum $k$ such that a graph can be isometrically embedded
into a $k$-lattice graph. As with tree dimension, median graphs have finite lattice dimension~\cite{ovchinnikov2011graphs,Eppstein2005585}.
Paths are exactly the median graphs of lattice dimension equal to 1. We are led to the following definition.

For an $i\geq 1$, an \emph{$i$-lattice decomposition} of $G$ is a median decomposition $\DDD=(M,\XXX)$ satisfying (M1),(M2), where $M$ is a median graph of lattice dimension at most $i$.
We denote the set of $i$-lattice decompositions of $G$ as $\LLL^G_i$. The \emph{$i$-latticewidth} $\lw_i(G)$ of $G$ is the least width of any $i$-lattice decomposition of $G$:
$$\lw_i(G):=\min\limits_{\DDD\in \LLL^G_i}\max\{|X_a| \mid a\in V(M)\}.$$ 
Since $\LLL^G_i\subseteq \MMM^G_i$, we have $\mw_i(G)\leq \lw_i(G)$. The $1$-lattice decompositions are the path decompositions of $G$, therefore $\lw_1(G)=\pw(G)+1$. Similarly to the case of $\mw_i$,
the parameters $\lw_i$ form a hierarchy starting from pathwidth and converging to the clique number:
$$\pw(G)+1=\lw_1(G)\geq \lw_2(G)\geq \dots \geq \lw_{\infty}(G)=\mw(G)=\omega(G).$$

For a $k$-colourable graph $G$, let $c:V(G)\rightarrow \{1,\ldots,k\}$ be a proper colouring of $G$ and for $i=1,\ldots,k$, let $P_i$ be a path with $|c^{-1}(i)|$ many vertices, 
whose vertices are labeled by the vertices of $c^{-1}(i)$ with arbitrary order. Consider the $k$-lattice $L=\Box_{i=1}^{k}P_i$, 
whose vertices $\bold{a}=(v_1,\ldots,v_k)\in V(L)$ are labeled by the $k$-tuple of labels of $v_1,\ldots,v_k$. 
For a vertex $\bold{a}\in V(L)$, define $X_\bold{a}$ to be the set of vertices that constitute the $k$-tuple of labels of $\bold{a}$. Let $\XXX=(X_\bold{a})_{\bold{a}\in V(L)}$.

\begin{lemma}\cite{stavropoulos2015medianwidth}\label{chromdec}
 The pair $\DDD=(L,\XXX)$ is a median decomposition of $G$ of width $k$.
\end{lemma}

\noindent
Median decompositions obtained from a colouring of $V(G)$ as in Lemma~\ref{chromdec} are called \emph{chromatic median decompositions} and were introduced in Section~4 of~\cite{stavropoulos2015medianwidth}. 
Let $\chi(G)$ be the \emph{chromatic number} of $G$. An immediate corollary of Lemma~\ref{chromdec} is the following bound.

\begin{lemma}\label{lwchi}
 For any graph $G$, $\lw_{\chi(G)}\leq \chi(G)$.\qed
\end{lemma}

\noindent
The results and proofs in the rest of this Section are in the spirit of Section~5 in~\cite{stavropoulos2015medianwidth}. 
It was proven there that complete $i+1$-partite graphs strongly distinguish $\mw_{i+1}$ from $\mw_i$.

\begin{lemma}\cite{stavropoulos2015medianwidth}\label{ipartitemw}
 For every $i\geq 1$, $\mw_i(K_{n_1,\ldots,n_{i+1}})\geq \min_{j=1}^{i+1}\{n_j\}+1$, while $\mw_{i+1}(K_{n_1,\ldots,n_{i+1}})=i+1$.
\end{lemma}

\noindent
Consequently, this fact directly translates to the case of latticewidth parameters: complete $i+1$-partite graphs have unbounded $i$-latticewidth, but bounded $i+1$-latticewidth.

\begin{lemma}\label{ipartitelw}
 For every $i\geq 1$, $\lw_i(K_{n_1,\ldots,n_{i+1}})\geq \min_{j=1}^{i+1}\{n_j\}+1$, while $\lw_{i+1}(K_{n_1,\ldots,n_{i+1}})=i+1$.
\end{lemma}

\begin{proof}
 Let $K=K_{n_1,\ldots,n_{i+1}}$. Since $\chi(K)=\omega(K)=i+1$, Lemmata~\ref{omega<mw} and~\ref{lwchi} show that $\lw_{i+1}(K)=i+1$.
 
 On the other hand, recall that $\lw_i(K)\geq \mw_i(K)$. Lemma~\ref{ipartitemw} completes the proof.
\end{proof}

\noindent
The lattice dimension of partial cubes has been studied in~\cite{Eppstein2005585} and~\cite{cheng2012poset}, but we shall only need simpler versions of some of the machinery used there. To this end, we will also adopt a notation that is more fiting for our purposes.

For a $k$-lattice $L=\Box_{j=1}^kP^j$, 
let $\pi_j:\Box_{j=1}^kP^j\rightarrow P^j$ be the $j$-th projection of $L$ to its $j$-th factor $P^j$.
It is well known that we can always embed a median graph into a lattice
in an irredundant way, but let us fast prove this in the context our own notation.

\begin{lemma}\label{optemb}
 Let $k$ be the lattice dimension of a median graph $M$. Then there is an isometric embedding $\phi$ of $M$ 
 into a $k$-lattice $\Box_{j=1}^kP^j$ such that for every $j=1,\ldots,k$ and every $u_j\in V(P^j)$,
 $$\pi_j^{-1}(u_j)\cap \phi(V(M))\neq\emptyset.$$
\end{lemma}

\begin{proof}
 Let $\phi:M\rightarrow L=\Box_{j=1}^kP^j$ be an isometric embedding into a $k$-lattice $L$ with $V(H)$ minimal. 
 Then, for every $j=1,\ldots,k$ and each of the two ends $l_j\in V(P^j)$ it must be $\pi_j^{-1}(l_j)\cap \phi(V(M))\neq\emptyset$,
 otherwise we can embed $M$ into $(\Box_{h\neq j}P^h)\Box (P^j-l_j)$, a contradiction to the choice of $L$.
 The Lemma~follows by the fact that $\phi(M)$ is a connected subgraph of $L$.
\end{proof}

We say that two $\Theta$-classes $F_{x_1x_2}, F_{x'_1x'_2}$ of a median graph $M$ \emph{cross} if $W_{x_ix_{3-i}}\cap W_{x'_jx'_{3-j}}\neq \emptyset$ for any $i,j=1,2$. 
Otherwise, if there is a choice $i,j \in \{1,2\}$ such that $W_{x_ix_{3-i}}\subseteq W_{x'_jx'_{3-j}}$ and $W_{x_{3-i}x_i}\subseteq W_{x'_{3-j}x'_j}$, we call $F_{x_1x_2}, F_{x'_1x'_2}$ \emph{laminar}.
Two $U$-sets are laminar if their adjacent $\Theta$-classes are laminar.

A \emph{$\Theta$-system} of $M$ is a set of $\Theta$-classes of it. 
We call a $\Theta$-system $\FFF$ of $M$ a \emph{strong direction} in $M$ if all of its members are pairwise laminar
and for every $F_{a_1b_1},F_{a_2b_2},F_{a_3b_3}\in \FFF$, if $W_{a_1b_1}\subseteq W_{a_2b_2}$ and $W_{a_1b_1}\subseteq W_{a_3b_3}$, then $W_{a_2b_2}\subseteq W_{a_3b_3}$ or $W_{a_3b_3}\subseteq W_{a_2b_2}$ 
(or in other words, if there is a $\subseteq$-chain containing a $W$-set from every pair of complementary $W$-sets corresponding to the $\Theta$-classes of $\FFF$, with their complementary $W$-sets forming a $\supseteq$-chain).
For a mapping $\psi:G\rightarrow H$ and an edge $e\in E(H)$,
by $\psi^{-1}(e)$ we mean $\{uv\in E(G)\mid \psi_j(u)\psi_j(v)=e\}.$
Notice that for a $k$-lattice $H=\Box_{j=1}^kP^j$, the family $\{\pi_j^{-1}(e_j)\mid e_j\in E(P^j)\}$ is a strong direction in $H$. 
When embedded into a lattice, a median graph inherits in a natural way the lattice's strong directions.

\begin{lemma}\label{embedding}
 Let $\phi:M \rightarrow L$ be an isometric embedding of a median graph $M$ into a $k$-lattice $L=\Box_{j=1}^kP^j$ as in Lemma~\ref{optemb}. 
 Then for every $j=1,\ldots,k$ the following are true:
 \begin{enumerate}[(i)]
  \item for every $e_j\in E(P^j)$,  $\phi^{-1}(\pi_j^{-1}(e_j))$ is a $\Theta$-class of $M$
  \item the family $\Sigma_j=\{\phi^{-1}(\pi_j^{-1}(e_j))\mid e_j\in E(P^j)\}$ is a strong direction of $M$
  $\phi^{-1}(\pi_j^{-1}(e_j))$ is a subset of $\phi^{-1}(\pi_j^{-1}(u_j))$.
 \end{enumerate}
\end{lemma}

\begin{proof}
 \begin{enumerate}[(i)]
  \item Let $e_j=u_jv_j$. Since $\pi_j^{-1}(u_j), \pi_j^{-1}(v_j)$ are complementary $W$-sets of $L$ and $\phi$ is an isometry,
  we have that $\phi^{-1}(\pi_j^{-1}(u_j)), \phi^{-1}(\pi_j^{-1}(v_j))$ are complementary $W$-sets of $M$. Since $\phi^{-1}(\pi_j^{-1}(e_j))$ is the set of edges between them, 
  they constitute a $\Theta$-class of $M$.
  \item Follows from (i) and the fact that $\{\pi_j^{-1}(e_j)\mid e_j\in E(P^j)\}$ is a strong direction in $L$.
 \end{enumerate}

\end{proof}

We call two separations $(U_1,U_2), (W_1,W_2)$ of a graph $G$ \emph{laminar} if there is a choice $i,j \in \{1,2\}$ such that $U_i\subseteq W_j$ and $U_{3-i} \supseteq W_{3-j}$, 
otherwise we say they \emph{cross}. A set of separations is called laminar if all of its members are pairwise laminar separations of $G$. We shall need the following lemma 
from~\cite{stavropoulos2015medianwidth}.

\begin{lemma}\cite{stavropoulos2015medianwidth}\label{lammed}
 Let $(M,\XXX)$ a median decomposition of $G$.  If the $\Theta$-classes $F_{ab}$, $F_{cd}$ are laminar in $M$, 
 then the corresponding separations $(Y_{ab},Y_{ba})$ and $(Y_{cd},Y_{dc})$ are laminar in $G$.
\end{lemma}
 
We are ready to present the analogue of Theorem~\ref{imwchar}, which roughly says that the $i$-latticewidth of a graph corresponds to the largest ``intersection'' of the best choice of
$i$ path decompositions of the graph. More specifically, in the following theorem let us denote path decompositions $(P^j,\ZZZ^j)$ with $\DDD^j=(P^j,\ZZZ^j)$.

\begin{theorem}\label{ilwchar}
 For any graph $G$ and any integer $i\geq 1$, $$\lw_i(G)=\min\limits_{\DDD^1,\ldots,\DDD^i\in \PPP^G}\max\{|\bigcap_{j=1}^iZ^j_{u_j}|\mid u_j\in V(P^j)\}.$$
\end{theorem}

\begin{proof}
 Let $$\lambda:=\min\limits_{\DDD^1,\ldots,\DDD^i\in \PPP^G}\max\{|\bigcap_{j=1}^iZ^j_{u_j}|\mid u_j\in V(P^j)\}.$$
 For $\DDD^1,\ldots,\DDD^i\in \PPP^G$, consider the pair $(L,\XXX)$, where $L=\Box_{j=1}^iP^j$ and $X_{(u_1,\ldots,u_i)}=\bigcap_{j=1}^iZ^j_{u_j}$.
 Note that (T1) for $\DDD^1,\ldots,\DDD^i$ implies (M1) for $(L,\XXX)$. Clearly, for every $v\in V(G)$, we have
 $$X^{-1}(v)=\Box_{j=1}^iZ^{j^{-1}}(v),$$ so by Lemma~\ref{cartmed}, (M2) also holds. 
 Then $(L,\XXX)$ is a valid $i$-lattice decomposition of $G$, therefore $$\lw_i(G) \leq \max\{\bigcap_{j=1}^iZ^j_{u_j}\mid u_j\in V(P^j)\}.$$
 Since $\DDD^1,\ldots,\DDD^i$ were arbitrary, it follows that $\lw_i(G)\leq \lambda$. 
 
 For the opposite implication, consider an $i$-lattice decomposition $(M,\XXX)$ of $G$ of width $\lw_i(G)$. 
 Let $k\leq i$ be the lattice dimension of $M$ and let $\phi:M\rightarrow L=\Box_{j=1}^kP^j$ be an isometric embedding as per Lemma~\ref{optemb}.
 By Lemma~\ref{embedding}(i),(ii), each $$\Sigma_j=\{\phi^{-1}(\pi_j^{-1}(e_j))\mid e_j\in E(P^j)\}$$ is a strong direction in $M$. 
 By the definition of a strong direction and Lemma~\ref{lammed}, there are path decompositions $\DDD^j=(P^j,\ZZZ^j)$ of $G$
 obtained by each $\Sigma_j$ where for each $u_j\in V(P^j)$ we have
 $$Z_{u_j}^j=\bigcup_{\pi_j(\phi(a))=u_j}X_a.$$
 Every vertex of $L$ is exactly the intersection of all the sublattices of $L$ of codimension $1$ that contain it. 
 In other words, for each $a\in V(M)$, we have $$\{a\}=\bigcap_{\substack{\pi_j(\phi(a))=u_j\\ j=1,\ldots,k}}\phi^{-1}(\pi_j^{-1}(u_j)).$$  
 It follows that
 $$X_a=\bigcap_{\substack{\pi_j(\phi(a))=u_j\\ j=1,\ldots,k}}Z_{u_j}^j.$$
 The trasversals from $\ZZZ^1,\ldots,\ZZZ^i$ that can achieve maximal size for the intersection of its elements, clearly correspond to the elements of $\XXX$.
 For $\lambda$, consider the decompositions $\DDD^1,\ldots,\DDD^k$
 together with the trivial decomposition of $G$ consisting of one bag being the whole $V(G)$ and repeated $i-k$ times. Then
  $$\lambda \leq \max\{|\bigcap_{j=1}^{k}Z^j_{u_j}|\mid u_j\in V(P^j)\}=\max\{|X_a|\mid a\in V(M)\}=\lw_i(G).$$  
\end{proof}

\section{Cops and Robber}\label{sec:game}
In this section, we describe in detail the game we sketched in Section \ref{sec:intro}, where the robber player plays against $i$ cop players. 
As in the case of treewidth and pathwidth, we will examine two variations of the game: 
one where the robber is \emph{visible} to each cop player and one where the robber is \emph{invisible} to them. 

\subsection{$i$ Cop Players vs a Visible Robber}\label{subsec:visible}
Let us precisely describe the \emph{``$i$-Cops and visible Robber''} game on a graph $G$ with \emph{cooperation at most $k$}, played by the $i$ \emph{Cop players} and a visible \emph{Robber player}. 
Let $X\subseteq V(G)$. An \emph{$X$-flap} is the vertex set of a component of $G\setminus X$.

A \emph{position} of the game is an $i$-tuple of pairs $((Z^1,R^1),\ldots,(Z^i,R^i))$ where $Z^1,\ldots,Z^i,R^1,\ldots,R^i\subseteq V(G)$. A \emph{move} is a triple $(j,Z,R)$,
where $j\in\{1,\ldots,i\}$ and $Z,R\subseteq V(G)$. The initial position of the game is always of the form $((Z_0^1,R_0^1),\ldots,(Z_0^i,R_0^i))=((\emptyset,C),\ldots,(\emptyset,C)$, 
where $C$ is a connected component of $G$. 
A \emph{play} is a sequence of moves and their corresponding positions. 
The $l$-th \emph{round} of a play starts from a position $((Z_{l-1}^1,R_{l-1}^1),\ldots,(Z_{l-1}^i,R_{l-1}^i))$. 
The players then make a move and a new position $((Z_l^1,R_l^1),\ldots,(Z_l^i,R_l^i))$ is obtained according to the following steps:
\begin{enumerate}[(i)]
 \item the \emph{robber player} chooses a cop team $j\in \{1,\ldots,i\}$ with $R_{l-1}^j\nsubseteq Z_{l-1}^j$, to be the next to move
 (if $R_{l-1}^j\subseteq Z_{l-1}^j$, the $j$-th team can't be chosen by the robber)
 \item the \emph{$j$-th cop player} then chooses a new vertex set $Z_l^j$ for the $j$-th cop team, 
 based only on knowledge of sets $Z=Z_{l-1}^j$ and $R=R_{l-1}^j$ 
 (each cop player can't see though in which round the overall game is at any point, how many times they were chosen or what the other cop players have played so far, 
 but they remember what was their last choice $Z$ and the choice $R$ of the robber in the same round)
 \item if $R_{l-1}^j\subseteq Z_l^j$, the \emph{robber player} keeps the same $R_l^j=R_{l-1}^j$, 
 otherwise the robber chooses a $Z_l^j$-flap $R_l^j$ such that $R_{l-1}^j$, $R_l^j$ are subsets of a common $(Z_{l-1}^j\cap Z_l^j)$-flap
 and such that $R_l^j$ \emph{intersects} $\bigcap_{j'\neq j}R_{l-1}^{j'}$
 \item the move $(j,Z_l^j,R_l^j)$ \emph{updates} the position $((Z_{l-1}^1,R_{l-1}^1),\ldots,(Z_{l-1}^i,R_{l-1}^i))$ of the game
 into $((Z_l^1,R_l^1),\ldots,(Z_l^i,R_l^i))$ by replacing the $j$-th pair $(Z_{l-1}^j,R_{l-1}^j)$ with $(Z_l^j,R_l^j)$ and leaving the rest of the position as is, 
 namely for every $j'\neq j$, $(Z_l^{j'},R_l^{j'})=(Z_{l-1}^{j'},R_{l-1}^{j'})$.
\end{enumerate}

\noindent
Observe that by (iii), the set $\bigcap_{j=1}^iR_l^{j}$, which contains the vertices that the robber can occupy, is non-empty for every round $l$ of the game. The play ends when it arrives at a position $((Z^1,R^1),\ldots,(Z^i,R^i))$ with $|\bigcap_{j=1}^iZ^j|> k$, in which case the robber wins, 
or when it arrives at a position with $R^j\subseteq Z^j$ for every $j=1,\ldots,i$ and $|\bigcap_{j=1}^iZ^j|\leq k$, in which case the cop players win.
Otherwise, when the game never ends, the robber player also wins. 

If the cop players have a winning strategy in the above game, 
we say that ``$i$ teams of cops \emph{with vision} can search the graph with cooperation at most $k$''. 
If the cop players can, in addition, always win in such a way that $R_l^j\subseteq R_{l-1}^j$ for every $j=1,\ldots,i$ and every round $l$, 
we say that ``$i$ teams of cops \emph{with vision} can \emph{monotonely} search the graph with cooperation at most $k$''. 

Notice that for $i=1$, the game clearly becomes the classical
``Cops and Robber'' game with one cop player and a visible robber, which characterises treewidth (recall that treewidth is the $1$-medianwidth). 
Analogously, monotone winning strategies for the $i$ cop players characterise $i$-medianwidth.

\begin{theorem}\label{visible}
 A graph $G$ can be monotonely searched with cooperation at most $k$ by $i$ teams of cops \emph{with vision} if and only if $\mw_i(G)\leq k$. 
\end{theorem}

\begin{proof}
 Let $\mw_i(G)\leq k$. By Theorem~\ref{imwchar}, there are tree decompositions $(T^1,\ZZZ^1),\ldots,(T^i,\ZZZ^i)\in\TTT^G$ with 
 $$\max\{|\bigcap_{j=1}^iZ^j_{t_j}|\mid t_j\in V(T^j)\}\leq k.$$ For each $T^j$, choose an arbitrary root $r_j$ and 
 consider the respective partial order $\trianglelefteq^j$ obtained by the rooted tree $(T^j,r_j)$ with $r_j$ its $\trianglelefteq^j$-minimal element.
 For every $t_j\in V(T^j)$, let $$V_{t_j}^j:=\bigcup_{\substack{s_j\in V(T^j) \\ t_j\trianglelefteq^j s_j}}Z_{s_j}^j.$$
 Then the cop players have the following winning strategy, which is easily seen to be well-defined:
 \begin{itemize} 
  \item for every $j\in \{1,\ldots,i\}$, the $j$-th player always chooses bags of $(T^j,\ZZZ^j)$, when he is selected to move
  \item the first time the $j$-th cop player is chosen by the robber to move, he chooses $Z_{r_j}^j$
  \item suppose that on the last time the $j$-th cop moved, he chose $Z_{t_j}^j$ and the robber chose the $Z_{t_j}^j$-flap 
  that is a subset of $V_{s_j}^j$ for a unique child $s_j$ of $t_j$ in $T^j$. Then, next time he is selected by the robber to move,
  he chooses $Z_{s_j}^j$.
 \end{itemize}
 Clearly, by the properties of tree decompositions, the above strategy is monotone. The strategy is winning, because for every position $((Z^1,R^1),\ldots,(Z^i,R^i))$ of a play
 we have $Z^j\subseteq \ZZZ^j$ and hence, $|\bigcap_{j=1}^iZ^j|\leq k$.
 
 Conversely, suppose that the cop players have a monotone winning strategy $\sigma$. Since the cop players are invisible to each other, we can view $\sigma$ as $\sigma=(\sigma^1,\ldots,\sigma^i)$, 
 where projection $\sigma^j$ of $\sigma$ corresponds to the individual strategy of $j$-th player.
 Each $\sigma^j$ can be represented by a directed rooted tree $(\vec{T^j},r_j)$ as follows: $r_j$ is labeled with $Z_{r_j}=\emptyset$ and 
 its outgoing arcs are labeled with the vertex sets of the connected components of $G$.
 The rest of the nodes $t_j$ are labeled with $Z_{t_j}^j\subseteq V(G)$ corresponding to subsets of $V(G)$ occupied by cops
 of the $j$-th cop player and the rest of the arcs $(t_j,s_j)$ are labeled with $R_{(t_j,s_j)}^j\subseteq V(G)$ corresponding to possible (legal) moves of the robber player. 
 That is, for every pair of arcs $(t_j,s_j), (s_j, u_j)$ of $\vec{T^j}$, we have that $R_{(t_j, s_j)}^j$ is a $Z_{t_j}^j$-flap, $R_{(s_j, u_j)}^j$ is a $Z_{s_j}^j$-flap, 
 and $R_{(s_j, u_j)}^j, R_{(t_j, s_j)}^j$ are subsets of a common $(Z_{t_j}^j\cap Z_{s_j}^j)$-flap.
 
 Moreover, since $\sigma$ is monotone, for every pair of arcs $(t_j,s_j), (s_j,u_j)$ of $\vec{T^j}$, it must be $R_{(s_j,u_j)}^j\subseteq R_{(t_j, s_j)}^j$. Hence,
 for every arc $(t_j,s_j)$, $R_{(t_j,s_j)}^j$ is a $(Z_{t_j}^j\cap Z_{s_j}^j)$-flap, too (otherwise the robber can break the monotonicity condition). In other words, for every arc $(t_j,s_j)$,
 $(Z_{t_j}^j\cap Z_{s_j}^j)$ separates $R_{(t_j,s_j)}^j\cup Z_{s_j}^j$ from $V(G)\setminus R_{(t_j,s_j)}^j$.  It is easy to see that the satisfaction of the statement of Lemma~\ref{twprop} (iii), 
 combined with the fact that every vertex of the graph is in a $Z_{t_j}^j$ set, is a sufficient condition for the pair $(T^j,\ZZZ^j=(Z_{t_j}^j)_{t_j\in V(T^j)})$ to be a tree decomposition of $G$, 
 where $T^j$ is the underlying undirected tree of $\vec{T^j}$  (the fact that $Z_{r_j}^j=\emptyset$ does not hurt (T1),(T2), and this can even be easily lifted by contracting $r_j$ to one of its children and removing $Z_{r_j}^j$ from $\ZZZ^j$).
 
 Observe that by selecting appropriately the order in which the cop players play and her respective choice $R$ of each round, the robber can force all positions $((Z^1,R^1),\ldots,(Z^i,R^i))$, 
 where $(Z^1,\ldots,Z^i)$ can be any transversal from the families $\ZZZ^1,\ldots,\ZZZ^i$ and satisfying $\bigcap_{j=1}^iZ^j\neq \emptyset$, if the cop players play according to $\sigma$. 
 Since $\sigma$ is a winning strategy for the cop players, we have that $\max\{|\bigcap_{j=1}^iZ^j_{t_j}|\mid t_j\in V(T^j)\}\leq k$ and the proof is complete by Theorem~\ref{imwchar}.
\end{proof}

\subsection{$i$ Cop Players vs an Invisible Robber}\label{subsec:inv}

To describe the \emph{``$i$-Cops and invisible Robber''} game, where the robber is invisible to the cop players, we will need to state it in a slightly alternative fashion. 
Positions, moves, rounds and plays aredefined as in Section~\ref{subsec:visible}. The initial position of the game is always $((Z_0^1,R_0^1),\ldots,(Z_0^i,R_0^i))=((\emptyset,V(G)),\ldots,(\emptyset,V(G))$. 
As in the case of the visible robber, the $l$-th \emph{round} of a play starts from a position $((Z_{l-1}^1,R_{l-1}^1),\ldots,(Z_{l-1}^i,R_{l-1}^i))$. 
Compared to steps  (i)-(iv) from Section~\ref{visible}, the round plays as follows:

\begin{enumerate}[(a)]
 \item same as (i)
 
 \item same as (ii)
 
 \item if $R_{l-1}^j\subseteq Z_l^j$, the \emph{robber player} keeps the same $R_l^j=R_{l-1}^j$, 
 otherwise the robber is automatically assigned with $R_l^j$ being the set of all vertices connected to $R_{l-1}^j$ with a path in the graph $G\setminus(Z_{l-1}^j\cap Z_l^j)$
 
 \item same as (iv).
\end{enumerate}

\noindent
The winning conditions of the game are exactly the same as the ones of Section~\ref{subsec:visible}.

If the cop players have a winning strategy, 
we say that ``$i$ teams of cops \emph{without vision} can search the graph with cooperation at most $k$''. 
If the cop players can, in addition, always win in such a way that $R_{l+1}^j\subseteq R_l^j$ for every $j=1,\ldots,i$ and every round $l$, 
we say that ``$i$ teams of cops \emph{without vision} can \emph{monotonely} search the graph with cooperation at most $k$''. 

Pathwidth corresponds to the $1$-latticewidth and for $i=1$, the game becomes the classical
``Cops and Robber'' game with one cop player and an invisible robber, that characterises pathwidth. 
Similarly, monotone winning strategies for the $i$ cop players characterise $i$-latticewidth.

\begin{theorem}\label{invisible}
 A graph $G$ can be monotonely searched with cooperation at most $k$ by $i$ teams of cops \emph{without vision} if and only if $\lw_i(G)\leq k$. 
\end{theorem}

\begin{proof}
 The proof is a direct adaptation of the proof of Theorem~\ref{visible}. We still briefly sketch it for the sake of completeness. Let $\lw_i(G)\leq k$. By Theorem~\ref{ilwchar},
 there are path decompositions $(P^1,\ZZZ^1),\ldots,(P^i,\ZZZ^i)$ with 
 $$\max\{|\bigcap_{j=1}^iZ^j_{u_j}|\mid u_j\in V(P^j)\}\leq k.$$ For $j=1,\dots,i$, let $P^j=(u_1^j,\ldots,u_{n_j}^j)$. 
 If the $j$-th cop player plays successively $Z_{u_1^j}^j,\ldots,Z_{u_{n_j}^j}$ each time he is chosen by the robber, then the $i$ cop players win monotonely.
 
 Conversely, suppose that the cop players have a monotone winning strategy. Since the cop players are invisible to each other, 
 the overall strategy of the cop players comprises individual strategies of each cop player.
 For $j=1,\ldots,i$, this individual strategy of the $j$-th cop player can be viewed as a sequence $(Z_1^j,\ldots,Z_{n_j}^j)$, 
 which he will successively follow each time he is chosen to play again. Let $R_m^j$ be the set assigned to the robber player
 in step (c), after the $j$-th cop player has chosen $Z_m^j$ in step (b). By the definition of monotonicity, we have $R_m^j\subseteq R_{m-1}^j$. 
 By (c), this implies that $Z_{m-1}^j\cap Z_m^j$ separates $R_m^j\cup Z_m^j$ from $V(G)\setminus R_m^j$. 
 
 By letting $P^j=(u_1^j,\ldots,u_{n_j}^j)$ and $Z_{u_m^j}^j:=Z_m^j$, we can then easily see that $(P,\ZZZ^j=(Z_{u_m^j}^j)_{u_m^j\in P^j})\in \PPP^G$. 
 Even though the robber can choose any order with which the cop players will play and 
 force any position $((Z^1,R^1),\ldots,(Z^i,R^i))$ with $(Z^1,\ldots,Z^i)$ an arbitrary transversal of $(\ZZZ^1,\ldots,\ZZZ^i)$ satisfying $\bigcap_{j=1}^iZ^j\neq \emptyset$,
 the cop players still always win with cooperation at most $k$.
 By Theorem~\ref{ilwchar}, the path decompositions $(P^1,\ZZZ^1),\ldots,(P^i,\ZZZ^i)$ show that $\lw_i(G)\leq k$.
  
\end{proof}

\section{Concluding Remarks}\label{sec:conc}

All the structural open questions mentioned in~\cite{stavropoulos2015medianwidth} for $i$-medianwidth naturally translate to directions for further research 
for the case of the newly introduced $i$-latticewidth as well.

One of the most notable facts about the classical Cops and Robber game is that the cop player can search a graph with $k$ cops if and only if he can search it with $k$ cops \emph{monotonely}. 
The equivalence in the strength of non-monotone and monotone strategies in the classical Cops and Robber game is obtained either from knowledge of 
the obstructing notion for the respective width parameter, such as \emph{brambles} being the obstructions for small treewidth~\cite{seymour1993graph}, 
or by arguments making use of the \emph{submodularity} of an appropriate connectivity function (whose definition we omit), such as the size of the \emph{border} $\theta X$, 
the set of vertices in a vertex set $X$ adjacent to the complement of $X$ (for example, see~\cite{bienstock1991quickly,BIENSTOCK1991239}).

It is a fundamental question to see if non-monotone winning strategies for the cop players in the $i$-Cops and Robber game are stronger than monotone ones, 
unlike the case for $i=1$. However, we have no access yet to obstructing notions of $i$-medianwith of $i$-latticewidth, whose presence might provide certificates for winning strategies for the robber, 
as in the case of treewidth and pathwidth. To this end, we also don't know if the notion of submodularity can be properly adjusted to provide similar results for any $i>1$, 
in the fashion it does for $i=1$.

\bibliographystyle{abbrv}
\bibliography{CopsRobbers} 

\end{document}